\documentclass[11pt]{amsart}
\usepackage{amssymb}
\usepackage{amsmath}
\usepackage{tikz}

\newtheorem{theorem}{Theorem}[section]

\newtheorem{lemma}[theorem]{Lemma}

\theoremstyle{definition}

\newcommand{\asi}{\textnormal{asi}}
\newcommand{\res}{\upharpoonright}

\title{Definable K\H{o}nig theorems}
\author{Matt Bowen, Felix Weilacher}
\email{matthewbowen2015@gmail.com, fweilach@andrew.cmu.edu}

\begin{document}

\maketitle

\begin{abstract}
    Let $X$ be a Polish space with Borel probability measure $\mu,$ and let $G$ be a Borel graph on $X$ with no odd cycles and maximum degree $\Delta(G).$  We show that the Baire measurable edge chromatic number of $G$ is at most $\Delta(G)+1$, and if $G$ is $\mu$-hyperfinite then the $\mu$-measurable edge chromatic number obeys the same bound. More generally, we show that $G$ has Borel edge chromatic number at most $\Delta(G)$ plus its asymptotic separation index.
\end{abstract}

\section{Introduction}

A classic theorem of K\H{o}nig states that every bipartite graph $G$ admits a (proper) edge coloring using $\Delta(G)$ many colors.  While this result is generally stated for finite graphs only, the compactness principle ensures that the same bound holds even when $G$ is infinite.  However, as the compactness principle in this context is equivalent to the existence of non-pinciple ultrafilters, obtaining bounds on chromatic numbers this way produces colorings that are poorly behaved with respect to any topological or measure structure on the graph.

The field of descriptive combinatorics is roughly dedicated to the problem of obtaining graph theoretic results that do respect underlying topological and measure theoretic structure.  The typical objects of study are Borel graphs, i.e. graphs whose vertex sets are a Polish space $(X,\tau)$ and whose edge sets are Borel subsets of $X^2.$  This field essentially originated with the work of Kechris, Solecki, and Todorcevic in \cite{kst}, and many applications have been found in areas such as descriptive set theory, group dynamics, and finite combinatorics.

Given a Borel graph $G$ with compatible Borel probability measure $\mu,$ the Borel edge chromatic number, $\chi'_{B}(G),$ is the smallest cardinal $\kappa$ so that $G$ admits Borel proper edge coloring using $\kappa$ colors. The Baire measurable edge chromatic number, $\chi'_{BM}(G),$ is the minimum of $\chi'_B(G \res C),$ where $C$ varies over $\tau$-comeagre $G$-invariant Borel sets. The $\mu$-measurable edge chromatic number, $\chi'_\mu(G),$ is defined analogously.  

The study of these parameters was begun in \cite{kst}, where it was shown that the greedy upper bound $\chi'_B(G)\leq 2\Delta(G)-1$ holds, while the irrational rotation graph, $G_\alpha,$ gives a simple example of a $2$-regular acylic graph where $2=\chi'(G_\alpha)<\chi'_B(G_\alpha)=\chi_\mu(G_\alpha)=\chi_{BM}(G_\alpha)=3.$  Constructions of Marks in \cite{Marks} show that the greedy bound cannot always be improved, even in the acyclic case, in stark contrast with K\H{o}nig's bound.  In the special case of $\mu$-preserving graphs better positive results are known: Cs\'{o}ka, Lippner, and Pikhurko showed that bipartite $\mu$-invariant graphs satisfy $\chi'_\mu(G)\leq \Delta(G)+1$ in \cite{csoka.lippner.pikhurko}, and this was recently improved by Pikhurko and Greb\'{\i}k to work for not necessarily bipartite graphs in  \cite{pikhurko.grebik}, matching the optimal general bound in the finite case and improving on an earlier bound of Bernshteyn \cite{Bernshteyn.lll}.  In the Baire and not necessarily measure preserving settings, Marks \cite{Marks} showed that graphs with $\Delta(G)\leq 3$ satisfy $\chi'_{BM}(G),\chi'_\mu(G)\leq 4,$ and Kechris and Marks \cite{kechris.marks} showed that acyclic, $d$-regular graphs satisfy $\chi'_{BM}=d$ whenever $d\geq 3.$ In the Borel setting the edge chromatic numbers of Schreier graphs of abelian groups were very recently characterized in \cite{bht, grebik.rozhon, weilacher.abelian}, the latter of whom also proved a $\Delta(G)+1$ bound for some actions of polynomial growth groups, but few other general results seem to be known.  

\vspace{2mm}

In this paper we prove a Baire measurable version of K\H{o}nig's theorem. 

\vspace{2.5mm}

\begin{theorem}\label{BM konig}
Let $G$ be a bipartite Borel graph. Then $\chi'_{BM}(G)\leq \Delta(G)+1.$
\end{theorem}

\vspace{2mm}

We can also obtain the same result in the (not necessarily pmp) measure setting, so long as the graph is hyperfinite $\mu$-a.e., i.e. when some invariant conull induced subgraph of $G$ is a countable increasing union of Borel graphs with finite connected components (component finite).  Note that every locally countable Borel graph is hyperfinite on an invariant comeagre set, and that many naturally arising graphs, such as the Schreier graphs of free Borel actions of amenable groups, are hyperfinite a.e.

\vspace{2mm}

\begin{theorem}\label{mu konig}
Let $X$ be a Polish space with probability measure $\mu,$ and let $G$ be a $\mu$-hyperfinite bipartite Borel graph on $X$.  Then $\chi'_\mu(G)\leq \Delta(G)+1.$
\end{theorem}

\vspace{2mm}

Unfortunately there is no hope of proving the same bound in the Borel setting, as a result of Conley, Jackson, Marks, Seward, and Tucker-Drob \cite{CJMSTD} shows that there are hyperfinite, $d$-regular acyclic graphs with $\chi'_B(G)=2d-1$ for any $d.$  Still, in the special case of $2$-ended graphs we get the following strengthening:

\vspace{2mm}

\begin{theorem}\label{2 bip}
Let $G$ be a bipartite Borel graph whose connected components all have 2 ends.  Then $\chi'_B(G)\leq \Delta(G)+1.$
\end{theorem}

\vspace{2mm}

This problem was considered by Weilacher in \cite{weilacher}, where a bound of $\Delta(G)+1$ was proven for actions of $2$-ended groups and where it was conjectured that this bound suffices in general.  This confirms the conjecture for bipartite graphs.

The previous three theorems will follow from a single more general one which is stated in terms of asymptotic separation index, a parameter recently defined in \cite{dimension}. We say a locally finite Borel graph $G$ has asymptotic separation index $\leq s$, written $\asi(G) \leq s$, where $s \in \omega$, if for any $r \in \omega$ we can find Borel sets $U_0,\ldots,U_s$ which cover $X$ such that for each $i$, the graph $G^{\leq r} \res U_i$ is component finite, where $G^{\leq r}$ is the Borel graph on $X$ for which two points are adjacent if the path distance between them in $G$ is at most $r$. Our general theorem is:

\begin{theorem}\label{asi alt}
Let $G$ be a bipartite Borel graph with finite asymptotic separation index. Then $\chi_B'(G) \leq \Delta(G) + \asi(G)$.
\end{theorem}

In \cite{weilacher.abelian}, it was asked whether a bound of $\Delta(G)+1$ holds for general Borel graphs $G$ with $\asi(G) = 1$. Again, this confirms a positive answer for bipartite graphs.

The parameter $\asi$ has been used in similar results in \cite{dimension} (Theorem 8.1). The authors in that paper also note that it is open whether there is any $G$ with $1 <\asi(G) < \infty$. Theorem \ref{asi alt} provides further motivation to study this question.

Note that if $\Delta(G)$ is infinite, Theorem \ref{asi alt} is automatic by the Lusin-Novikov uniformization theorem. Thus for the rest of the paper we assume $\Delta(G)$ is finite. Note also that Theorem \ref{asi alt} gives the classical K\H{o}nig's theorem when $\asi(G) = 0$. The following results show that Theorem \ref{asi alt} immediately implies Theorems \ref{BM konig}, \ref{mu konig}, and \ref{2 bip}:

\begin{theorem}[Essentially \cite{conley.miller.toast},\cite{miller.2end}]
Let $G$ be a locally finite Borel graph on a Polish space $X$, and $\mu$ a Borel probability measure on $X$
\begin{enumerate}
    \item There is a $G$-invariant comeager Borel set $C \subset X$ such that $\asi(G \res C) \leq 1$.
    \item If $G$ is $\mu$-hyperfinite, there is a $G$-invariant $\mu$-conull Borel set $C \subset X$ such that $\asi(G \res C) \leq 1$.
    \item If each connected component of $G$ has 2 ends, $\asi(G) = 1$.
\end{enumerate}
\end{theorem}

Finally, we mention that determining if the measurable K\H{o}nig bound holds for bipartite Borel graphs that are not hyperfinite or pmp is still open.  By applying results on measurable fractional perfect matchings from \cite{bowen.kun.sabok} we may reduce the problem to determining if acyclic, leafless graphs admit measurable matchings that cover all but a sparse set of vertices.  In the hyperfinite case this reduces further to the $2$-ended case by applying Theorem A from \cite{conley.miller.pm}, which motivated the results in the present paper.  However, a recent counterexample of Kun \cite{gabor.new} shows that no analogue of \cite{conley.miller.pm} Theorem A holds for graphs that are not hyperfinite, so a different approach must be used.

\section{Proofs}

Graphs with finite asymptotic separation index can be approximated well by component finite subgraphs. A typical strategy for coloring them, therefore, is to define colorings on those components just using finite combinatorics, then somehow address global coherence. We will use the classical K\H{o}nig's theorem for the finite combinatorical part. The next lemma helps with coherence.

For $U$ a set of vertices and $r \in \omega$, define $B(U,r)$ to be the set of vertices with path distance from $U$ at most $r$, and $E(U,r)$ to be the set of vertices with path distance from $U$ exactly $r$.

\begin{lemma}\label{barriers alt}
Let $G$ be a locally finite Borel graph on $X$ with $\asi(G) \leq s \in \omega$. For any $\Delta \in \omega$, we can find Borel sets $S_i^j \subset X$ for $j < s$ and $i < \Delta$ such that, letting $S_i = \bigcup_j S_i^j$,
\begin{enumerate}
    \item The restrictions of $G$ to $B(S_{i}^j,2)$ for each $i,j$ and to $(X \setminus S_i)$ for each $i$ are component finite. 
    \item For each $j$ and $i \neq i'$, the path distance from $S_i^j$ to $S_{i'}^j$ is at least 5.
\end{enumerate}

\end{lemma}

\begin{proof}
Let $X = U_0 \sqcup \cdots \sqcup U_{s-1} \sqcup V$ be a Borel partition of $X$ witnessing $\asi(G) \leq s$ for $r = 10\Delta+1$. That is, such that the graphs $G^{\leq r} \res U_j$ for each $j$ and $G^{\leq r} \res V$ are component finite. For each $i < \Delta$ and $j < s$, set $S_{i}^j = E(U_j,5i)$. Each $S_i^j$ is Borel since each $U_j$ is, and the distance requirement (2) is clearly satisfied.

Fix $i$. We want to show $G \res (X \setminus S_i)$ is component finite. By K\H{o}nig's lemma (not theorem!), it suffices to show that if $x_0,x_1,\ldots$ is an injective path through $G$, then it contains a point of $S_i$. We first show that for all $j$, it contains a point of $X \setminus B(U_j,5\Delta)$. If not, then for all $n$, there is a $y_n \in U_j$ with distance at most $5\Delta$ from $x_n$. Then for all $n$, the distance between $y_n$ and $y_{n+1}$ is at most $10\Delta+1 = r$, so the $y_n$'s are all in the same $G^{\leq r} \res U_j$-component, so there are only finitely many of them. But $\{x_n\} \subset B(\{y_n\},5\Delta)$ is infinite, contradicting local finiteness.

By the same proof, our path contains some point of $X \setminus V$, say $x_{n_0} \in U_j$. Let $n$ be such that $x_n \not\in B(U_j,5\Delta)$. Consider the distance from $U_j$ to $x_m$ for $m$ between $n_0$ and $n$. When $m = n_0$, it is 0. When $m = n$, it is more than $5\Delta$. When $m$ changes by 1, it can change by at most 1. Thus since $0 \leq 5i < 5\Delta$, there must be some $m$ for which $x_m \in E(U_j,5i) = S_i^j \subset S_i$.

Fix $i,j$. We want to show that $G \res B(S_i^j,2)$ is component finite. Above, we saw that any infinite injective path contains a point of distance more than $5\Delta$ from $U_j$. All points of $B(S_i^j,2)$ have distance at most $5i+2 < 5\Delta$ from $U_j$, though, so we are again done by K\H{o}nig's lemma.
\end{proof}

The proof of the classical K\H{o}nig's theorem proceeds by successively removing matchings from the initial graph, each time lowering the maximum degree. Our approach is similar, except that the edge sets removed are not quite matchings. The following lemma is our inductive step. 

We say a vertex is covered by a set of edges (e.g. a matching) if it is contained in some edge in the set.

\begin{lemma}\label{inductive alt}
Let $G$ be a bipartite Borel graph on $X$ with $\Delta(G) \leq \Delta \in \omega$. Let $s \in \omega$. Let $S^j \subset X$ for $j < s$ such that, letting $S = \bigcup_j S^j$, the restrictions of $G$ to $X \setminus S$ and to $B(S^j,2)$ for each $j$ are component finite. Then there are Borel matchings $M,N^0,\ldots,N^{s-1} \subset G$ such that each vertex of degree $\Delta$ is covered by $M \cup \bigcup_j N^j$, and for each $j$, each edge of $N^j$ is contained in $B(S^j,2)$.
\end{lemma}

\begin{proof}
Consider the component finite bipartite Borel graphs $G \res (X \setminus S)$ and $G \res B(S^j,2)$ for $j < s$. By the classical K\H{o}nig's theorem, each component of each of these graphs admits a matching which covers its vertices of maximal degree. By the Lusin-Novikov uniformization theorem, there is a Borel way of picking such a matching for each component. Thus, let $M \subset G \res (X \setminus S)$ be a Borel matching covering each vertex of $G \res (X \setminus S)$-degree $\Delta$, and likewise for $N^j$ and $G \res B(S^j,2)$ for each $j$.

Let $x$ have $G$-degree $\Delta$; we need to show $x$ is covered by $M \cup \bigcup_j N^j$. If $x \in B(S^j,1)$ for some $j$, then $x$ still has degree $\Delta$ in $G \res B(S^j,2)$, so it is covered by $N^j$. Else, $x$ has degree $\Delta$ in $G \res (X \setminus S)$, so it is covered by $M$.
\end{proof}

We now prove Theorem \ref{asi alt}:

\begin{proof}
Let $\Delta = \Delta(G) \in \omega$, $s = \asi(G) \in \omega$, and $X = V(G)$. Let $S_i^j \subset X$ for $i < \Delta$, $j < s$ be as given by Lemma \ref{barriers alt}. Let $G_0 = G$.

Suppose $i < \Delta$ and we have Borel $G_i \subset G$ with maximum degree $\leq \Delta - i$. Apply Lemma \ref{inductive alt} to $G_i$ and the $S_i^j$'s to get Borel matchings $M_i,N_i^0,\ldots,N_i^{s-1}$ whose union covers each vertex of $G_i$-degree $\Delta - i$ and with each edge of $N_i^j$ contained in $B(S_i^j,2)$ for each $j$. Let $G_{i+1} = G_i \setminus (M_i \cup \bigcup_j N_i^j)$. Then $G_{i+1}$ is Borel and has maximum degree $\leq \Delta - i - 1$, so we can repeat this process.

At the end, $G_\Delta$ has maximum degree $\leq 0$, and so is empty. Thus, the matchings we removed along the way, $M_i$ and $N_i^j$ for $i < \Delta$, $j < s$, union to all of $G$. Furthermore, for each $j$, since $B(S_i^j,2)$ and $B(S_{i'}^j,2)$ are disjoint for $i \neq i'$, $\bigcup_i N_i^j$ is still a Borel matching, call it $N_j$. Now the $M_i$'s and $N_j$'s give a Borel $(\Delta + s)$-edge coloring of $G$.
\end{proof}

\section*{Acknowledgements}

The second author was partially supported by the ARCS foundation, Pittsburgh chapter.

\bibliographystyle{amsalpha} 
\bibliography{main}

\end{document}